\documentclass[a4paper,twoside]{amsart}
\usepackage[utf8]{inputenc}
\usepackage[hidelinks]{hyperref}
\usepackage{graphicx}
\usepackage{enumitem}

\theoremstyle{plain}
\newtheorem{theorem}{Theorem}[section]
\newtheorem{corollary}[theorem]{Corollary}
\newtheorem{prop}[theorem]{Proposition}
\newtheorem{observation}[theorem]{Observation}
\newtheorem{lemma}[theorem]{Lemma} 
\newtheorem{conjecture}[theorem]{Conjecture}
\newtheorem{question}[theorem]{Question}

\theoremstyle{definition}
\newtheorem{definition}[theorem]{Definition}
\newtheorem{example}{Example}
\newtheorem*{remark*}{Remark}
\newtheorem*{claim*}{Claim}
\theoremstyle{remark}
\newtheorem{remark}[theorem]{Remark}

\def\dom{\mathop{\mathrm{Dom}}\nolimits}

\def\str#1{\mathbf {#1}}
\def\Fraisse{Fra\"{\i}ss\' e}
\def\Aut{\mathop{\mathrm{Aut}}\nolimits}
\def\tp{\mathop{\mathrm{tp}}\nolimits}

\def\Semig{{\mathfrak M}}
\def\ESemig{\Semig=(M,\oplus,\mleq)}

\def\mleq{\preceq}

\def\mgeq{\succeq}
\def\nmgeq{\not\succeq}
\def\mgt{\succ}
\def\mlt{\prec}

\def\F{\mathbb{F}}

\def\Ind#1#2{#1\setbox0=\hbox{$#1x$}\kern\wd0\hbox to 0pt{\hss$#1\mid$\hss}
\lower.9\ht0\hbox to 0pt{\hss$#1\smile$\hss}\kern\wd0}
\def\Notind#1#2{#1\setbox0=\hbox{$#1x$}\kern\wd0\hbox to
0pt{\mathchardef\nn="0236\hss$#1\nn$\kern1.4\wd0\hss}\hbox 
to 0pt{\hss$#1\mid$\hss}\lower.9\ht0
\hbox to 0pt{\hss$#1\smile$\hss}\kern\wd0}
\def\ind{\mathop{\mathpalette\Ind{}}}
\def\nind{\mathop{\mathpalette\Notind{}}}

\def\pocs{partially ordered commutative semigroup}
\def\Aclass{\mathcal A^\delta_{K_1,K_2,C_1,C_2}}

\begin{document}
\bibliographystyle{alpha}

\title[]{Simplicity of the automorphism groups of generalised metric spaces}

\authors{
\author[D. M. Evans]{David M. Evans}
\address{Department of Mathematics\\ Imperial College London\\ London SW7~2AZ\\ UK.}
\email{david.evans@imperial.ac.uk}
\author[J. Hubi\v cka]{Jan Hubi\v cka}
\address{Department of Applied Mathematics (KAM)\\ Charles University\\ Prague, Czech Republic}
\email{hubicka@kam.mff.cuni.cz}
\author[M. Kone\v cn\'y]{Mat\v ej Kone\v cn\'y}
\address{Department of Applied Mathematics (KAM)\\ Charles University\\ Prague, Czech Republic}
\email{matej@kam.mff.cuni.cz}
\author[Y. Li]{Yibei Li}
\address{Imperial College London\\ London SW7~2AZ\\ UK.}
\email{yibei.li16@imperial.ac.uk}
\author[M. Ziegler]{Martin Ziegler}
\address{Mathematisches Institut\\ Albert--Ludwigs--Universität Freiburg\\ D-79104 Freiburg, Germany}
\email{ziegler@uni-freiburg.de}}
\date{}
\thanks{Jan Hubi\v cka and Mat\v ej Kone\v cn\'y are supported by project 18-13685Y of the Czech Science Foundation (GA\v CR) and by a project that has received funding from the European Research Council (ERC) under the European Union’s Horizon 2020 research and innovation programme (grant agreement No 810115). Mat\v ej Kone\v cn\'y is supported by the Charles University project GA UK No 378119. Yibei Li is supported by President’s scholarship from Imperial College. Jan Hubi\v cka is supported by the Center for Foundations of Modern Computer Science (Charles University project UNCE/SCI/004).}

\begin{abstract}
Tent and Ziegler proved that the automorphism group of the Urysohn sphere is simple and that the automorphism group of the Urysohn space is simple modulo bounded automorphisms. A key component of their proof is the definition of a stationary independence relation (SIR). In this paper we prove that the existence of a SIR satisfying some extra axioms is enough to prove simplicity of the automorphism group of a countable structure. The extra axioms are chosen with applications in mind, namely homogeneous structures which admit a ``metric-like amalgamation'', for example all primitive 3-constrained metrically homogeneous graphs of finite diameter from Cherlin's list.
\end{abstract}
\maketitle
\section{Introduction}
In 2011, Macpherson and Tent~\cite{Macpherson2011b} proved that the automorphism groups of \Fraisse{} limits of free amalgamation classes are simple. This was followed by two papers of Tent and Ziegler~\cite{Tent2013,Tent2013b} where they prove that the isometry group of the Urysohn space (the unique complete separable homogeneous metric space universal for all finite metric spaces) modulo bounded isometries (i.e. isometries $f$ with a finite bound on the distance between $x$ and $f(x)$) is simple and that the isometry group of the Urysohn sphere is simple. Later, Evans, Ghadernezhad and Tent~\cite{Evans2016} proved simplicity for automorphism groups of some Hrushovski constructions, and Li~\cite{Li2018} proved simplicity for the structures from Cherlin's list of 26 primitive triangle-constrained homogeneous structures with 4 binary symmetric relations (see appendix of~\cite{Cherlin1998}). 

More recently, Tent and Ziegler's method was generalised to asymmetric structures. Li~\cite{Li2019} proved that the automorphism groups of some of Cherlin's asymmetric structures in the appendix of~\cite{Cherlin1998} are simple. The same result for non-trivial linearly ordered free homogeneous structures has been proved independently by Calderoni, Kwiatkowska and Tent~\cite{calderoni2020simplicity} and Li~\cite{Li2020}. Also in~\cite{Li2020}, simplicity was proved for the automorphism groups of the universal $n$-linear orders for $n\geq 2$. Another recent example where (non-stationary) independence relations have been used to prove strong results about automorphism groups of structures is a paper by Kaplan and Simon~\cite{Kaplan2019}.

In this paper, we adapt the methods of Tent and Ziegler and prove the following theorem (definitions and examples will be given in the upcoming paragraphs).

\begin{theorem}
\label{thm:main}
Let $\F$ be a transitive countable relational structure with a bounded 1-supported metric-like stationary independence relation $\ind$. Then $\Aut(\F)$ is simple.
\end{theorem}

As direct corollaries of Theorem~\ref{thm:main}, we get the following two more concrete results, for which the definitions will be given in Section~\ref{sec:corollaries}.

\begin{theorem}\label{thm:semigroups}
Let $\ESemig$ be a finite archimedean \pocs{} with at least two elements and let $\F$ be a homogeneous $\Semig$-metric space which realises all distances. Assume that $\F$ admits an $\Semig$-shortest path independence relation $\ind$ and that $\ind$ is a 1-supported SIR. Then $\Aut(\F)$ is simple.
\end{theorem}

\begin{theorem}\label{thm:cherlin}
If $G$ is a countably infinite metrically homogeneous graph which corresponds to one of the primitive 3-constrained finite-diameter classes from Cherlin's catalogue~\cite{Cherlin2011b}, then $\Aut(G)$ is simple.
\end{theorem}

\subsection{Stationary independence relations}
The notion of stationary independence relations (Definition~\ref{SIR}) was developed by Tent and
Ziegler in their paper on the Urysohn space~\cite{Tent2013}.
It has several generalisations (e.g. for structures with closures~\cite{Evans2016}), but for our purposes the original variant suffices.

Let $\F$ be a relational structure and let $A,B\subseteq \F$ be finite subsets.
We will identify them with the substructures induced by $\F$ on $A$ and $B$ respectively and by $AB$ we will denote the union $A\cup B$ (and hence also the substructure induced by $\F$ on $AB$). If the set $A=\{a\}$ is singleton, we may write $a$ instead of $\{a\}$. Uppercase letters will denote sets while lowercase will denote the elements of the structure, which we call \emph{vertices} owing to the combinatorial background of part of the authors. As is usual in this area, if $A\subseteq \F$, we sometimes assume that it has some implicit enumeration. This is clear from the context and should not cause any confusion.

Let $A, X\subseteq \F$. By the \emph{type of $A$ over $X$} (denoted by $\tp(A/X)$) we mean the orbit of $A$ under the action of the stabilizer subgroup of $\Aut(\F)$ with respect to $X$. If $p=\tp(A/X)$, we say that $B\subseteq\F$ \emph{realises} $p$ (and denote it as $B\models p$) if $B$ lies in $p$, in other words, if there is an automorphism of $\F$ fixing $X$ pointwise which maps $A$ to $B$. To simplify the notation, we write $\tp(A)$ for $\tp(A/\emptyset)$. Our types correspond to realised types in a (strongly) homogeneous structure in the standard model-theoretic terminology. In fact, we may assume that the language is chosen so that $\F$ is homogeneous, that is, partial automorphisms between finite substructures of $\F$ extend to automorphisms.

\begin{definition}[Stationary Independence Relation]
\label{SIR}
Let $\mathbb{F}$ be a relational structure.  A ternary relation $\ind$ on finite subsets of $\mathbb{F}$ is called a \emph{stationary independence
relation} (\emph{SIR}, with $A\ind_{C}B$ being pronounced ``$A$ is independent from $B$ over $C$'') if the following conditions are satisfied:
\begin{enumerate}[label=SIR\arabic*]
 \item\label{invariance} \emph{(Invariance)}. The independence of finite subsets of $\mathbb{F}$ only depends on their type. In particular, for 
every
automorphism $f$ of $\mathbb{F}$, we have $A\ind_{C}B$ if and only if
$f(A)\ind_{f(C)}f(B)$.
 \item\label{symmetry}\emph{(Symmetry)}. If $A\ind_C B$, then $B\ind_C A$.
\item\label{monotonicity} \emph{(Monotonicity)}. If $A\ind_{C}BD$, then
$A\ind_{C}B$ and $A\ind_{BC}D$.
\item\emph{(Existence)}. \label{existence} For every $A,B$ and $C$ in $\mathbb{F}$, there
is some
$A'\models \tp(A/C)$ with $A'\ind_{C}B$. 
\item\emph{(Transitivity)} \label{transitivity} If $A\ind_C B$ and $A\ind_{BC} B'$, then $A\ind_C B'$.
\item\emph{(Stationarity)} \label{stationarity} If $A$ and $A'$ have the same
type over $C$ and are both independent over $C$ from some set $B$ then they
also have the same type over $BC$. 
\end{enumerate}
\end{definition}
Note that by an observation of Baudisch~\cite{Baudisch2016}, these axioms are redundant as Monotonicity can be derived from the rest of them. Stationary independence relations correspond to ``canonical amalgamations'' by putting $A\ind_C B$ if and only if the canonical amalgamation of $AC$ and $BC$ over $C$ is isomorphic to $ABC$. The notion of canonical amalgamations can be formalised, see~\cite{Aranda2017}.

To make our proofs shorter, we will sometimes use Symmetry, Monotonicity and Existence implicitly. The following observation which follows from Invariance will be useful later.
\begin{observation}\label{obs:invariance}
If $\F$ is a relational structure, $\ind$ a SIR on $\F$ and $A\ind_C B$, then $A\ind_C BC$.
\end{observation}

\begin{definition}[$k$-supported SIR]
Let $k$ be a positive integer. We say that a SIR $\ind$ is \emph{$k$-supported} if for every $a,b,C$ such that $a\ind_C b$ there is $C'\subseteq C$ such that $\vert C'\vert\leq k$ and $a\ind_{C'} b$.
\end{definition}

\begin{observation}
For $k=1$, $k$-supportedness is equivalent to:

\emph{(1-supportedness)} If $a\ind_C b$ and $C = C_1\cup C_2$ then $a\ind_{C_1} b$ or $a\ind_{C_2} b$.
\end{observation}

We say that a structure $\F$ is \emph{transitive} if $\tp(a)=\tp(b)$ for every $a,b\in \F$.

\begin{definition}[Metric-like SIR]\label{defn:metriclike}
Let $\F$ be a relational structure with a SIR $\ind$. We say that $\ind$ is \emph{metric-like} if the following conditions are satisfied:
\begin{enumerate}
\item\label{metriclike_weird} If $a\notin A$, then $a\nind_A a$.
\item For every $a\in \F$ there is $b\in \F$ such that $a\neq b$ and $a\nind_\emptyset b$.
\item \emph{(Perfect triviality)}  If $A\ind_C B$ and $C\subseteq C'$ then $A\ind_{C'} B$.
\end{enumerate}
\end{definition}


\begin{lemma}\label{lem:metricity}
Let $\F$ be a relational structure with a SIR $\ind$ which satisfies Perfect triviality. Then $\ind$ satisfies
\begin{enumerate}
\item\emph{(Metricity)} If $A\ind_{C_1C_2} B$ and $C_1\ind_D B$ then $A\ind_{C_2D} B$.
\item\emph{(Triviality)} If $A\ind_B C$ and $A\ind_B D$ then $A\ind_B CD$.
\end{enumerate}
\end{lemma}
\begin{proof}
First assume that $A\ind_{C_1C_2} B$ and $C_1\ind_D B$. By Perfect triviality, $C_1\ind_{C_2D} B$ and $A\ind_{C_1C_2D} B$. Using Transitivity it follows that $A\ind_{C_2D}B$, which proves Metricity.

Now assume that $A\ind_B C$ and $A\ind_B D$. By Perfect triviality we get $A\ind_{BC} D$ and by Observation~\ref{obs:invariance} and Monotonicity it then follows that that $A\ind_{BC} CD$. Using Transitivity together with $A\ind_B C$ then implies $A\ind_B CD$.
\end{proof}
In fact, Metricity is equivalent to Perfect triviality if $\ind$ is a SIR. The following is a simple corollary of Triviality which will be useful later.
\begin{corollary}\label{cor:triviality}
If $a\ind_\emptyset x$ for every $x\in X$, then $a\ind_\emptyset X$.
\end{corollary}

\begin{definition}[Geodesic sequence]
Let $\F$ be a relational structure with a SIR $\ind$. We say that a sequence $a_1,\ldots,a_n\in \F$ of pairwise distinct vertices of $\F$ is \emph{geodesic} if for every $1\leq i<j<k\leq n$ it holds that $a_i\ind_{a_j}a_k$.
\end{definition}

\begin{definition}\label{defn:bounded}
Let $\F$ be a relational structure with a SIR $\ind$. We say that $\ind$ is \emph{bounded} if it satisfies

\medskip

\emph{(Boundedness)} There exists an integer $k_0$ such that if $a_0,\ldots,a_k$ is a geodesic sequence with $k \geq k_0$, then $a_0\ind_\emptyset a_k$.

\medskip

We denote the smallest such $k_0$ by $\|\ind\|$.
\end{definition}

The reader is encouraged to have the following examples in mind when reading this paper.
\begin{example}
Let $\F$ be the \Fraisse{} limit of all finite metric spaces using only distances $\{0,1,\ldots,n\}$ for some fixed $n\geq 2$ (clearly, one can view a metric space as a relational structure by introducing a binary relation for every distance). Define $\ind$ on $\F$ by putting $A\ind_C B$ if and only if for every $a\in A$ and every $b\in B$ it holds that $d(a,b) = \min(\{n\}\cup \{d(a,c)+d(b,c) : c\in C\})$. It is straightforward to check that $\ind$ is a bounded 1-supported metric-like SIR with $\|\ind\| = n$. 
\end{example}

For the Urysohn sphere, the only axiom which we do not have at hand is, paradoxically, Boundedness.
\begin{example}\label{ex:urysohn}
Let $\mathbb U_1$ be the Urysohn sphere, that is, the unique homogeneous separable complete metric space with distances from $[0,1]$ which is universal for all finite metric spaces with distances from $[0,1]$. We will denote its metric by $d$. Define the relation $\ind$ on finite subsets of $\mathbb U_1$ by putting $A\ind_C B$ if and only if for every $a\in A$ and every $b\in B$ it holds that $d(a,b) = \min(\{1\} \cup \{d(a,c)+d(b,c) : c\in C\})$. One can check that $\ind$ is a $1$-supported metric-like SIR, but does not satisfy Boundedness, as for every $k$ one can find a geodesic sequence with $k+1$ vertices such that the distance of every consecutive pair of them is smaller that $\frac{1}{k}$.
\end{example}



\begin{example}[$k$-supported metric-like SIR]\label{ex:non1sup}
Let $k\geq 1$ and $n\geq 3$ be integers. Put $S=\{1,\ldots,n\}^k\cup \{0\}^k$, let $A$ be a set and let $d\colon A^2\to S$ be a function. Let $\mleq$ be the product order on $S$ (i.e. $(a_1,\ldots,a_k)\mleq (b_1,\ldots,b_k)$ if and only if $a_i\leq b_i$ for every $1\leq i\leq k$) and let $\oplus$ be the component-wise addition on $S$ capped at $n$ (i.e. $(a_1,\ldots,a_k)\oplus (b_1,\ldots,b_k) = (c_1,\ldots,c_k)$, where $c_i=\min(n, a_i+b_i)$ for every $1\leq i\leq k$). 

We say that $(A,d)$ is an $[n]^k$-metric space if the following holds for every $x,y,z\in A$:
\begin{enumerate}
\item $d(x,y)=d(y,x)$,
\item $d(x,y) = (0,\ldots,0)$ if and only if $x=y$,
\item $d(x,z) \mleq d(x,y)\oplus d(y,z)$.
\end{enumerate}

One can verify that the class of all finite $[n]^k$-metric spaces is a \Fraisse{} class. Consider the structure $\mathbb M_k=(M_k,d)$, which is the \Fraisse{} limit of the class of all $[n]^k$-metric spaces, and define $\ind$ on $\mathbb M_k$ by putting $A\ind_C B$ if and only if for every $a\in A$ and every $b\in B$ it holds that $d(a,b)=\inf_\mleq \{d(a,c)\oplus d(c,b) : c\in C\}$. As $\mleq$ has a maximum, the infimum of the empty set is $(n,\ldots,n)$.

It is easy to verify that $\ind$ is a bounded metric-like SIR. Moreover, it is $k$-supported, but not $k'$-supported for any $k'<k$, which is witnessed by vertices $a,b, c_1,  \ldots, c_k\in \mathbb M_k$ such that $a\ind_{\{c_1,\ldots, c_k\}} b$, $d(a, c_i) = (1,\ldots,1)$ for every $i$ and $d(b, c_i)$ is equal to $1$ on the $i$-th coordinate and equal to $2$ everywhere else.
\end{example}

\section{Geodesic sequences}
In this section we prove some auxiliary results about geodesic sequences which will be used later. Fix a transitive relational structure $\F$ with a metric-like SIR $\ind$.

\begin{lemma}\label{lem:geodesic}
Let $a_1, \ldots, a_n$ be a geodesic sequence of vertices of $\F$ and let $b\in \F\setminus \{a_n\}$. Then there is $a_{n+1}\models \tp(b/a_n)$ such that $a_1,\ldots,a_{n+1}$ is a geodesic sequence.
\end{lemma}
\begin{proof}
Using Existence, pick $a_{n+1}\models\tp(b/a_n)$ such that $a_1\cdots a_{n-1} \ind_{a_n} a_{n+1}$. Consider any $1\leq i < j \leq n-1$. By Monotonicity, $a_i\ind_{a_n}a_{n+1}$ and hence, by Perfect triviality, $a_i\ind_{a_ja_n} a_{n+1}$. Since $a_1,\ldots,a_n$ is a geodesic sequence, we know that $a_i\ind_{a_j}a_n$. Transitivity now implies that $a_i\ind_{a_j}a_{n+1}$ and hence $a_1,\ldots,a_{n+1}$ is a geodesic sequence.
\end{proof}

\begin{lemma}\label{lem:exists1_aux}
Let $a,b,c\in \F$ be distinct such that $a\ind_\emptyset b$. There is a geodesic sequence $a_0, a_1, \ldots, a_n\in \F$ satisfying the following:
\begin{enumerate}
\item $a=a_0$ and $b=a_n$, and
\item for every $0\leq i\leq n-1$ it holds that $\tp(a_ia_{i+1}) = \tp(ac)$,
\item $n = \|\ind\|$.
\end{enumerate}
\end{lemma}
\begin{proof}
First observe that since all vertices have the same type, for every $v\in \F$ there is $v'\in \F$ such that $\tp(vv')=\tp(ac)$. Put $n=\|\ind\|$ and use Lemma~\ref{lem:geodesic} repeatedly to obtain a geodesic sequence $a,x_1,\ldots,x_{n}$ such that all consecutive pairs of vertices have the type $\tp(ac)$. We know that $a\ind_\emptyset x_{n}$. By Stationarity, $\tp(x_{n}/a)=\tp(b/a)$, hence there exists an automorphism $f$ of $\F$ which fixes $a$ and maps $x_{n}$ to $b$. By Invariance, $f(a),f(x_1),\ldots,f(x_{n})$ has the desired properties.
\end{proof}

\begin{lemma}\label{lem:geodesic_type}
Let $v_1,\ldots,v_k$ and $w_1,\ldots,w_k$ be geodesic sequences of vertices of $\F$ such that for every $1\leq i<k$ we have $\tp(v_iv_{i+1})=\tp(w_iw_{i+1})$. Then $\tp(v_1\cdots v_k)=\tp(w_1\cdots w_k)$.
\end{lemma}
\begin{proof}
We shall prove by induction on $m$ that $\tp(v_1\cdots v_m) = \tp(w_1\cdots w_m)$. For $m=2$ this is true by the assumption. Assume now that the statement is true for some $m$. Using the fact that $v_1,\ldots,v_k$ and $w_1,\ldots,w_k$ are geodesic sequences and Triviality we get that $v_1\cdots v_{m-1}\ind_{v_m}v_{m+1}$ and $w_1\cdots w_{m-1}\ind_{w_m}w_{m+1}$. By the assumption we have $\tp(v_mv_{m+1}) = \tp(w_mw_{i+m})$, hence Stationarity together with Invariance and the induction hypothesis give $\tp(v_1\cdots v_{m+1}) = \tp(w_1\cdots w_{m+1})$.


\end{proof}

\begin{prop}\label{prop:concat_geodesic}
Let $a,b,c$ be vertices of $\F$ satisfying the following:
\begin{enumerate}
\item $a\ind_b c$,
\item there is a geodesic sequence $a=v_1, \ldots, v_k=b$,
\item there is a geodesic sequence $b=w_1, \ldots, w_\ell=c$.
\end{enumerate}
Then there is a geodesic sequence $a=x_1,\ldots,x_{k+\ell-1}=c$ such that $\tp(x_1\cdots x_k)=\tp(v_1\cdots v_k)$ and $\tp(x_k\cdots x_{k+\ell-1})=\tp(w_1\cdots w_\ell)$.
\end{prop}
\begin{proof}
Use Lemma~\ref{lem:geodesic} and the fact that all vertices have the same type $\ell-1$ times to extend $v_1,\ldots,v_k$ by vertices $w_2',\ldots,w_\ell'$ such that $v_1,\ldots,v_k,w_2',\ldots,w_\ell'$ is a geodesic sequence and for every $1\leq i<\ell$ we have $\tp(w_i'w_{i+1}')=\tp(w_iw_{i+1})$, where we put $w_1'=v_k$ to simplify the notation.

In particular, $w_1',\ldots,w_\ell'$ is a geodesic sequence. Using Lemma~\ref{lem:geodesic_type} we get that $\tp(w_1\cdots w_\ell)=\tp(w_1'\cdots w_\ell')$, so in particular $\tp(w_1w_\ell) = \tp(w_1'w_\ell')$. Since $w_1=w_1'=v_k$, we have that $\tp(w_\ell/v_k)=\tp(w_\ell'/v_k)$. By the hypothesis and the construction, $w_\ell\ind_{v_k} v_1$ and $w_\ell'\ind_{v_k} v_1$. Stationarity implies that $w_\ell'\models\tp(w_\ell/v_1v_k)$, so in particular $w_\ell'\models\tp(w_\ell/v_1)$.

In other words, there is an automorphism $g$ of $\F$ such that $g(v_1)=v_1$ and $g(w_\ell')=w_\ell$. The image of $v_1,\ldots,v_k,w_2',\ldots,w_\ell'$ under $g$ then gives the desired geodesic sequence $x_1,\ldots,x_{k+\ell-1}$.
\end{proof}

Let $a,b\in \F$ be distinct. We say that \emph{$b$ is almost free from $a$} if $a\nind_\emptyset b$ and for every $c\in \F$ different from $a,b$ such that $a\ind_b c$ it holds that $a\ind_\emptyset c$.

\begin{observation}\label{obs:almost_free_preserved}
Let $a,b\in \F$ be such that $b$ is almost free from $a$. For every $a',b'\in\F$ such that $\tp(a'b')=\tp(ab)$ it holds that $b'$ is almost free from $a'$.
\end{observation}

\begin{lemma}\label{lem:maximal_distance}
Suppose that $\ind$ is bounded. For every $a\in \F$ and every finite $X\subseteq \F$ such that $a\notin X$ there is $b\in \F$ such that $a$ is almost free from $b$, $b$ is almost free from $a$, and $b\ind_a X$. In particular, $b\nind_\emptyset a$ and $b\ind_\emptyset X$.
\end{lemma}
\begin{proof}
We claim that there exist $a',b'\in \F$ such that $b'$ is almost free from $a'$ and $a'$ is almost free from $b'$. Suppose that this is true. Since $\F$ is transitive, there is an automorphism $f$ such that $f(a')=a$. Pick $b\models\tp(f(b')/a)$ such that $b\ind_a X$. By Observation~\ref{obs:almost_free_preserved}, $b$ is almost free from $a$ and $a$ is almost free from $b$. The ``in particular'' part is immediate using Corollary~\ref{cor:triviality}.





Hence it suffices to prove the claim. Pick $a',b'\in \F$ such that $b'\nind_\emptyset a'$ and the length of the longest geodesic sequence starting at $a'$ finishing at $b'$ is as large as possible. (As $\ind$ is bounded, such $a',b'$ exist.) Pick $c\in\F$ such that $a'\ind_{b'} c$. By Proposition~\ref{prop:concat_geodesic}, we can extend the geodesic sequence from $a'$ to $b'$ by some $c'\models\tp(c/b')$. By the properties of $a',b'$ we get that $a'\ind_\emptyset c'$. Invariance and Stationarity then imply that $a'\ind_\emptyset c$ and consequently $b'$ is almost free from $a'$.

To prove that $a'$ is almost free from $b'$, pick $c\in\F$ such that $b'\ind_{a'} c$. Since the reverse of a geodesic sequence is a geodesic sequence, we extend the geodesic sequence from $b'$ to $a'$ by some $c'\models\tp(c/a')$ as above. Suppose that $b'\nind_\emptyset c'$. Since $\F$ is transitive, there is an automorphism $f$ such that $f(b')=a'$. The image of the geodesic sequence from $b'$ to $c'$ is then a geodesic sequence starting at $a'$ which is longer than the geodesic sequence from $a'$ to $b'$ we started with. This is a contradiction, hence $b'\ind_\emptyset c'$. As before, we get that $a'$ is almost free from $b'$ which concludes the proof.
\end{proof}


\section{Proof of Theorem~\ref{thm:main}}
We will closely follow the proof from the Tent--Ziegler paper on the Urysohn sphere~\cite{Tent2013b} and use the following result by Tent and Ziegler~\cite{Tent2013}.

\begin{definition}
Let $\F$ be a countable structure with a stationary independence relation $\ind$, let $g\in \Aut(\F)$, let $A\subseteq \F$ be finite and let $p=\tp(a/A)$ be a type. We say that \emph{$g$ moves $p$ almost maximally} if there is a realisation $x\models p$ such that $$x\ind_A g(x).$$
\end{definition}

\begin{theorem}[Corollary~5.4,~\cite{Tent2013}]\label{thm:tz}
Let $\F$ be a countable structure with a stationary independence relation and let $g$ be an automorphism of $\F$ which moves every type over every finite set almost maximally. Then every element of $\Aut(\F)$ is a product of sixteen conjugates of $g$.
\end{theorem}

Throughout the section, we fix $\F$ and $\ind$ as in Theorem~\ref{thm:main} ($\F$ is a transitive countable relational structure with a bounded 1-supported metric-like stationary independence relation $\ind$) and put $G = \Aut(\F)$. As before, we may assume that $\F$ is homogeneous (this will slightly simplify the proof of Lemma~\ref{lem:pump}).

\begin{lemma}\label{lem:exists1}
If $g\in G$ is not the identity then there is $a\in\F$ and $h\in G$ which is a product of $\|\ind\|$ conjugates of $g$ such that $a\ind_\emptyset h(a)$.
\end{lemma}
\begin{proof}
Let $a\in \F$ be such that $a\neq g(a)$ and pick $b\in \F$ such that $a\ind_\emptyset b$ (Existence). Use Lemma~\ref{lem:exists1_aux} to obtain a geodesic sequence $a=a_0,\ldots,a_n=b$ such that $n=\|\ind\|$ and for every $0\leq i\leq n-1$ we have $\tp(a_ia_{i+1})=\tp(ag(a))$. This means that there are automorphisms $h_0,\ldots, h_{n-1}$ such that $h_i(a)=a_i$ and $h_{i}(g(a))=a_{i+1}$. Then $h_igh_i^{-1}$ moves $a_i$ to $a_{i+1}$ and the statement follows.
\end{proof}

\begin{lemma}\label{lem:movesall1}
Let $g\in G$ be such that for some $a\in \F$ we have $a\ind_\emptyset g(a)$. Then for every finite set $A\subset \F$ there is $x\in \F$ with $x\ind_\emptyset A$ and $x\neq g(x)$.
\end{lemma}
\begin{proof}
We may assume that $a\in A$. Put $Y=A\cup g^{-1}(A)$ and choose $b\in \F$ with $b\neq a$ and $b\nind_\emptyset a$ ($\ind$ is metric-like) such that moreover $b\ind_a Y$ (Existence and Invariance). This means that $b\notin g^{-1}(A)$ (if $b\in g^{-1}(A)$, then $b\in Y$, so $b\ind_a b$, which is in contradiction with part~(\ref{metriclike_weird}) of Definition~\ref{defn:metriclike}) and hence $g(b)\notin A$. We know that $a \ind_\emptyset g^{-1}(a)$ (by Invariance) and also $b\ind_a g^{-1}(a)$,
thus $b \ind_\emptyset g^{-1}(a)$ (Transitivity) and so $g(b)\ind_\emptyset a$ (Invariance). This means that $b\neq g(b)$ and therefore $g(b)\notin A\cup \{b\}$.

Use Lemma~\ref{lem:maximal_distance} to obtain $x\in \F$ such that $x\nind_\emptyset g(b)$ and $x\ind_\emptyset Ab$. By Monotonicity, $x\ind_\emptyset A$ and $x\ind_\emptyset b$, hence also $g(x)\ind_\emptyset g(b)$, thus $x\neq g(x)$.
\end{proof}

Let $X\subset \F$ be a finite set and let $a\in \F$ be such that $a\ind_\emptyset X$. We call the type $\tp(a/X)$ a \emph{free type}. (It is the unique such type over $X$.)


\begin{lemma}\label{lem:movesall}
Let $g\in G$ be such that for every free type $p$  there is a realisation $a\models p$ with $g(a)\neq a$. Then for every finite $X\subset \F$ and every type $q=\tp(x/X)$ with $x\notin X$, there is a realisation $c\models q$ such that $g(c)\neq c$.
\end{lemma}
\begin{proof}
Let $a$ be a vertex such that $a\ind_\emptyset X$ and $g(a)\neq a$ ($a$ exists by the assumptions of this lemma) and let $b\models q$ be such that $b\ind_X g(a)$.

If $b\nind_\emptyset g(a)$ then pick $c\models q$ such that $c\ind_X ag(a)$. This means that $c\nind_\emptyset g(a)$ (by Stationarity and Invariance) and $c\ind_\emptyset a$ (by Transitivity), giving us $g(c)\neq c$.

So we have $b \ind_\emptyset g(a)$. Use Lemma~\ref{lem:maximal_distance} to obtain $a'\in \F$ such that $a'\nind_\emptyset b$, $a'\ind_\emptyset X$, and $a'$ is almost free from $b$. By Stationarity, we have that $a\models \tp(a'/X)$, hence there is $f\in G$ fixing $X$ pointwise such that $f(a')=a$. Put $c'=f(b)$. In particular, $c'\models q$, $a\nind_\emptyset c'$, and $a$ is almost free from $c'$ (Observation~\ref{obs:almost_free_preserved}).

Choose $c\models \tp(c'/Xa)$ such that $c\ind_{Xa} g(a)$. In particular, $c\nind_\emptyset a$ (Invariance). By Observation~\ref{obs:almost_free_preserved}, $a$ is almost free from $c$. Using 1-suppor\-ted\-ness, $c\ind_{Xa} g(a)$ implies that either $c\ind_a g(a)$ (in which case $c\ind_\emptyset g(a)$ and hence $g(c)\neq c$), or $c\ind_X g(a)$. In this case we know that $\tp(c/X) = \tp(b/X)$ and $b\ind_X g(a)$ (using Perfect triviality on $b\ind_\emptyset g(a)$), hence by Stationarity and Invariance, $c\ind_\emptyset g(a)$, thus again $g(c)\neq c$.
\end{proof}

We say that $g\in G$ \emph{moves type $p$ by distance $k$} if there is $a\models p$ and a geodesic sequence $a=a_0,\ldots,a_k=g(a)$. If $p=\tp(x/X)$ is a type and $h$ is an automorphism or a partial automorphism defined on a finite set such that $X\subseteq \dom(h)$, we denote $h(p) = \tp\left(h'(x)/h'(X)\right)$, where $h'$ is some automorphism of $\F$ extending $h$ (remember that we assumed that $\F$ is homogeneous).


\begin{lemma}\label{lem:pump}
Let $g\in G$ be such that $g$ moves all types almost maximally or by distance $n$. Then there exists $h\in G$ such that $[g,h]=g^{-1}h^{-1}gh$ moves all types almost maximally or by distance $2n$.
\end{lemma}
\begin{proof}
As in~\cite{Tent2013b}, we construct $h$ by a ``back-and-forth'' construction as the union of a chain of finite partial automorphisms. We show the following: Let $h'$ be already defined on a finite set $U$ and let $p=\tp(x/X)$ be a type. Then $h'$ has an extension $h$ such that $[g,h]$ moves $p$ almost maximally or by distance $2n$.

We can assume that $X\cup g^{-1}(X)\subseteq U$. Put $V=h'(U)$. Let $a'$ be a realisation of $p$ such that $a'\ind_X Ug^{-1}(U)$ and let $b'$ be a realisation of $h'(\tp(a'/U))$ (which is a type over $V$). By the hypothesis on $g$ there are realisations $a\models \tp(a'/Ug^{-1}(U))$ and $b\models \tp(b'/V)$ such that either $a\ind_{Ug^{-1}(U)} g(a)$, or there is a geodesic sequence $a=a_0,\ldots,a_n=g(a)$ and similarly for $b$. We also have
$$
a\ind_X Ug^{-1}(U) \text{ and } b\ind_{h'(X)} V.
$$

Let $h_0$ be the isomorphism $Ua\simeq Vb$ and let $c$ be a realisation of $h_0^{-1}(\tp(g(b)/Vb))$ (which is a type over $Ua$) such that $c\ind_{Ua} g(a)$. Put $h$ to be the isomorphism $Uac\simeq Vbg(b)$. Observe that $[g,h](a) = g^{-1}(c)$. It remains to prove that $a$ witnesses that $[g,h]$ moves $p$ almost maximally or by distance $2n$.

Since $a\ind_X g^{-1}(U)$, we know that $g(a)\ind_{g(X)}U$. Using Metricity, we get
$$c\ind_{g(X)a}g(a),$$
thus from 1-supportedness we know that either $c\ind_a g(a)$ or $c\ind_{g(X)}g(a)$. In the second case we get $g^{-1}(c)\ind_X a$, which implies that $[g,h]$ moves $p$ almost maximally. Hence we can assume that
$$c\ind_a g(a).$$

By the choice of $a$ and $b$ we know that one of the following cases occurs:
\begin{enumerate}
\item First suppose that there are geodesic sequences $b=b_0,\ldots,b_n=g(b)$ and $g(a)=a_0,\ldots,a_n=a$ (the reverse of a geodesic sequence is a geodesic sequence by Symmetry). From the construction we know that $\tp(ac)=\tp(bg(b))$. This implies that there is a geodesic sequence $a=c_0, \ldots, c_n=c$. Since $g(a)\ind_a c$, Proposition~\ref{prop:concat_geodesic} gives a geodesic sequence starting at $g(a)$ and finishing at $c$ using $2n+1$ vertices (including $c$ and $g(a)$). Finally, taking the image of this sequence under $g^{-1}$ gives a geodesic sequence starting at $a$ and finishing at $g^{-1}(c)=[g,h](a)$  using $2n+1$ vertices. This means that $a$ witnesses that $[g,h]$ moves $p$ by distance $2n$.

\item Now assume that $a\ind_{Ug^{-1}(U)} g(a)$. Then in fact we have $a\ind_X g(a)$, because $a\ind_X Ug^{-1}(U)$ (Metricity). As $U\supseteq Xg^{-1}(X)$, $a\ind_X U$ also implies $g(a)\ind_{g(X)} X$ (by Invariance and Monotonicity), which together with $a\ind_X g(a)$ implies $a\ind_{g(X)} g(a)$ (Metricity). Thus from $c\ind_{a}g(a)$ we get $c\ind_{g(X)}g(a)$ (yet again Metricity) and thus $g^{-1}(c)\ind_X a$, i.e. $a$ witnesses that $[g,h]$ moves $p$ almost maximally.

\item Otherwise we have $b\ind_V g(b)$. Using that $h$ is an isomorphism of $Uac$ and $Vbg(b)$ and Invariance we obtain $a\ind_U c$. Then we get $a\ind_X c$, because $a\ind_X U$ (Metricity), and then, combining with $c\ind_a g(a)$ using Metricity again, we obtain $c\ind_X g(a)$. As in the previous case, $a\ind_X U$ implies $g(a)\ind_{g(X)} X$ and hence $c\ind_{g(X)} g(a)$, or $g^{-1}(c)\ind_X a$, i.e. $a$ witnesses that $[g,h]$ moves $p$ almost maximally.
\end{enumerate}
\end{proof}

Now we prove the following proposition, Theorem~\ref{thm:main} is then its direct consequence.
\begin{prop}\label{prop:main}
Let $\F$ be a countable relational structure with a bounded 1-supported metric-like stationary independence relation $\ind$ and let $g$ be a non-identity automorphism of $\F$. Then there is an automorphism of $\F$ which is a product of at most $2\|\ind\|^2$ conjugates of $g$ and $g^{-1}$ and moves every type over every finite set almost maximally.
\end{prop}
\begin{proof}
From Lemma~\ref{lem:exists1} we get an automorphism $g_0$ which is a product of at most $\|\ind\|$ conjugates of $g$ such that there is $a\in \F$ with $a\ind_\emptyset g_0(a)$. Using Lemma~\ref{lem:movesall1} we get that in fact for every free type there is a realisation which is not fixed by $g_0$.

Let $p=\tp(x/X)$ be a type. Either $x\in X$ (then $x\ind_X g(x)$, hence $g_0$ moves $q$ almost maximally), or $x\notin X$ and thus by Lemma~\ref{lem:movesall} there is a realisation of $p$ which is not fixed by $g_0$. This means that $g_0$ moves all types almost maximally or by distance 1.

Put $n=\lceil\log_2(\|\ind\|)\rceil$ and construct a sequence $g_0,g_1,\ldots, g_n$ of automorphisms of $\F$ using Lemma~\ref{lem:pump} such that every $g_i$ moves all types almost maximally or by distance $2^i$, and if $i\geq 1$ then $g_i$ is a product of two conjugates of $g_{i-1}$ and $g_{i-1}^{-1}$. For $g_{n}$ we get that it moves every type almost maximally or by distance at least $\|\ind\|$. In the latter case, we have for every type $p$ a realisation $a\models p$ and a geodesic sequence $a=a_0,\ldots,a_k=g(a)$, where $k\geq \|\ind\|$. Boundedness~(Definition~\ref{defn:bounded}) implies that $a\ind_\emptyset g(a)$, i.e. $g_n$ moves $p$ almost maximally, and hence $g_n$ moves all types almost maximally.

By the construction, $g_n$ is a product of at most $2^{\lceil\log_2(\|\ind\|)\rceil}$ conjugates of $g_0$ and $g_0^{-1}$, hence a product of at most $2^{\lceil\log_2(\|\ind\|)\rceil}\|\ind\| \leq 2\|\ind\|^2$ conjugates of $g$ and $g^{-1}$.
\end{proof}

\begin{proof}[Proof of Theorem~\ref{thm:main}]
Let $g$ be a non-identity automorphism of $\F$. We need to prove that if $N$ is a normal subgroup of $G$ such that $g\in N$, then $N=G$. If $g\in N$, then clearly $g^{-1}\in N$. Let $h\in G$. By Proposition~\ref{prop:main} and Theorem~\ref{thm:tz}, we know that $h$ can be written as a product of conjugates of $g$ and $g^{-1}$, hence $h\in N$. This is true for every $h\in G$, hence $N=G$ and $G$ is simple.
\end{proof}

\section{Corollaries}\label{sec:corollaries}
In this section we prove Theorems~\ref{thm:semigroups} and~\ref{thm:cherlin}.

\subsection{Semigroup-valued metric spaces}\label{sec:semigroups}
We say that a tuple $\ESemig$ is a \emph{\pocs} if the following hold:
\begin{enumerate}
\item $(M,\oplus)$ is a commutative semigroup,
\item $(M, \mleq)$ is a partial order which is reflexive ($a\mleq a$ for every $a\in M$),
\item for every $a, b\in M$ it holds that $a\mleq a\oplus b$, and
\item for every $a,b,c\in M$ it holds that if $b\mleq c$ then $a\oplus b \mleq a\oplus c$ ($\oplus$ is monotone with respect to $\mleq$).
\end{enumerate}
$\Semig$ is \emph{archimedean} if for every $a,b\in \Semig$ there is an integer $n$ such that $n\times a\mgeq b$, where by $n\times a$ we mean
$$\underbrace{a\oplus a \oplus \cdots \oplus a}_\text{$n$ times}.$$
Note that if $\Semig$ is archimedean and non-trivial, it follows that $\Semig$ does not have an identity.

Let $L$ be a set. An \emph{$L$-edge-labelled graph} is a tuple $\str A = (A, E, d)$, where $E\subseteq {A\choose 2}$ and $d$ is a function $E\to L$. Clearly, the set $E$ can be inferred from the function $d$ and thus we omit it. For simplicity, we write $d(x,y)$ instead of $d(\{x,y\})$ and we put $d(x,x)=0$, where $0$ is a symbol which is not an element of $\Semig$. When convenient, we naturally understand $0$ as the neutral element with respect to $\oplus$ and as the minimum element of $\mleq$.

We say that $\str A$ is \emph{complete} if the graph $(A,E)$ is a complete graph. Note that an $L$-edge-labelled graph can equivalently be viewed as a relational structure with an irreflexive binary symmetric relation $R^m$ for every $m\in L$ such that every pair of vertices is in at most one relation.

For a \pocs{} $\ESemig$, a complete $\Semig$-edge-labelled graph $\str A = (A, d)$ is an \emph{$\Semig$-metric space} if for every triple $a,b,c\in A$ of distinct vertices it holds that $d(a,b)\mleq d(a,c)\oplus d(b,c)$ (the \emph{triangle inequality}).

\medskip

Let $\F$ be an $\Semig$-metric space. We say that $\F$ \emph{admits an $\Semig$-shortest path independence relation} if for every $a,b\in \F$ and $C\subseteq \F$ finite we have that $\{d(a,c)\oplus d(c,b) : c\in C\}$ has an infimum with respect to $\mleq$ (note that $C$ can be empty which implies that $\Semig$ has maximum $\inf_\mleq(\emptyset)$). If $\F$ admits an $\Semig$-shortest path independence relation, then its \emph{$\Semig$-shortest path independence relation} is a ternary relation $\ind$ defined on finite subsets of $\F$ by putting $A\ind_C B$ if and only if for every $a\in A$ and every $b\in B$ it holds that $d(a,b)=\inf_\mleq\{d(a,c)\oplus d(c,b) : c\in C\}$.

Generalising concepts of Sauer~\cite{Sauer2012}, Conant~\cite{Conant2015} (see also~\cite{Hubicka2017sauerconnant}) and Braunfeld~\cite{Sam} (see also~\cite{Kabil2018}), Hubi\v cka, Kone\v cn\'y and Ne\v set\v ril~\cite{Konecny2018b,Hubicka2017sauer} introduced the framework of semigroup-valued metric spaces, which served as a motivation for this paper. Given a \pocs{} $\ESemig$ and a ``nice'' family $\mathcal F$ of $\Semig$-edge-labelled cycles, the structures of interest are $\Semig$-metric spaces which moreover contain no homomorphic images of members of $\mathcal F$. We will denote the class of all such finite structures $\mathcal M_\Semig^\mathcal F$.

The conditions of $\mathcal F$ are strong enough that one can then prove that $\mathcal M_\Semig^\mathcal F$ is a strong amalgamation class, its \Fraisse{} limit admits an $\Semig$-shortest path independence relation which is a SIR (provided that $\Semig$ has a maximum, otherwise one can still get a \textit{local} SIR), it has EPPA (for background, see~\cite{Hubicka2018EPPA,Siniora2}) and a precompact Ramsey expansion (for background, see~\cite{Hubicka2016,NVT14}), but they are general enough that most known binary symmetric homogeneous structures can be viewed as such a semigroup-valued metric space. In fact, it is conjectured that every primitive transitive homogeneous structure in a finite binary symmetric language with trivial algebraic closures admits such an interpretation (Conjecture~1 in~\cite{Konecny2018b}).

Now we are ready to prove Theorem~\ref{thm:semigroups}.
\begin{proof}[Proof of Theorem~\ref{thm:semigroups}]
We need to prove that $\ind$ is metric-like and bounded. (In fact, we do not need 1-supportedness for this, we only need it later in order to apply Theorem~\ref{thm:main}.)

Since $\F$ is homogeneous, all vertices have the same type. As $d(x,y)=0$ if and only if $x=y$ and $\oplus$ is monotone with respect to $\mleq$, it follows that if $a\notin A$, then $a\nind_A a$. The fact that there are $a\neq b\in \F$ such that $a\nind_\emptyset b$ follows from Stationarity, the fact that $\Semig$ has at least two elements (remember that $0\notin \Semig$) and the fact that $\F$ realises all distances.

Suppose now that $a\ind_C b$. If there was $c'\in \F\setminus C$ such that $a\nind_{Cc'} b$, this would mean that $\inf_\mleq\{d(a,c)\oplus d(c,b) : c\in C\cup\{c'\}\} \mlt \{d(a,c)\oplus d(c,b) : c\in C\} = d(a,b)$, hence $d(a,c')\oplus d(c',b)\nmgeq d(a,b)$, in other words, $abc'$ violates the triangle inequality which is a contradiction. Consequently, $\ind$ satisfies Perfect triviality and hence $\ind$ is metric-like.

\medskip

Next we prove that $\ind$ is bounded. Denote by $1$ the maximum element of $\Semig$ ($\Semig$ is finite and hence there is such an element). Assume that there are $a,b\in \Semig$ such that $a\oplus b = a$. This means (by associativity) that $a\oplus (n\times b) = a$ for every $n$. Let $c\in \Semig$ be arbitrary. By archimedeanity there is $n$ such that $n\times b \mgeq c$. But then $a = a\oplus (n\times b) \mgeq c$. Hence $a\mgeq c$ for every $c\in \Semig$, that is, $a=1$. In other words, for every $a,b\in \Semig\setminus \{1\}$ it holds that $a\oplus b\mgt a$, which implies that whenever $a_1,\ldots,a_{\lvert\Semig\rvert}\in \Semig$, then
$$\bigoplus_{i=1}^{\lvert\Semig\rvert} a_i = 1.$$

We can use this observation to prove that $\|\ind\| \leq \lvert\Semig\rvert$. Indeed, if $a_0,\ldots,a_{\lvert\Semig\rvert}$ is a geodesic sequence, we know that $d(a_0,a_{i+1})=d(a_0,a_i)\oplus d(a_i,a_{i+1})$. Using induction we get that
$$d(a_0,a_{\lvert\Semig\rvert})=d(a_1,a_2)\oplus d(a_2,a_3) \oplus \cdots \oplus d(a_{\lvert\Semig\rvert-1},a_{\lvert\Semig\rvert}),$$
that is, $d(a_0,a_{\lvert\Semig\rvert})$ is a sum of $\lvert\Semig\rvert$ elements of $\Semig$ and hence $d(a_0,a_{\lvert\Semig\rvert})=1$, which means that indeed $a_0\ind_\emptyset a_{\lvert\Semig\rvert}$.

We have proved that $\ind$ is bounded and metric-like, hence we can apply Theorem~\ref{thm:main} to show that $\Aut(\F)$ is simple.
\end{proof}

Note that whenever $\mleq$ is a linear order, the corresponding $\Semig$-shortest path independence relation is necessarily 1-supported. The following theorem is a direct consequence of this fact, Theorem~\ref{thm:semigroups} and existing results on semigroup-valued metric spaces~\cite{Konecny2018b,Hubicka2017sauer}.

Let $S\subseteq \mathbb R^+$ be a finite subset of positive reals such that the following operation $\oplus_S\colon S^2\to S$ is associative:
$$a\oplus_S b = \max\{x\in S : x \leq a+b\}.$$
Delhomm{\'e}, Laflamme, Pouzet, and Sauer~\cite{Delhomme2007} studied and Sauer later classified~\cite{Sauer2013,sauer2013b} such subsets. Ramsey expansions for all such classes of $(S, \oplus_S, \leq)$-metric spaces were obtained by Hubi\v cka and Ne\v set\v ril~\cite{Hubicka2016}, and Hubička, Kone\v cn\'y, Ne\v set\v ril and Sauer~\cite{HubickaSauerSmetric} (Nguyen Van Th\'e~\cite{NVT2009} earlier proved some partial results). We contribute to the study of such classes by the following result:
\begin{theorem}
Let $S\subseteq \mathbb R^+$ be a finite subset of positive reals such that $\Semig_S=(S, \oplus_S, \leq)$ is an archimedean \pocs{}. Then the automorphism group of the \Fraisse{} limit of the class of all finite $\Semig_S$-metric spaces is simple.
\end{theorem}

\subsection{Metrically homogeneous graphs}
A \emph{metrically homogeneous graph} is a graph whose path-metric is a homogeneous metric space. Cherlin~\cite{Cherlin2011b,Cherlin2013} gave a list of such graphs by describing the corresponding amalgamation classes of metric spaces. The vast majority of the list is occupied by the 5-parameter classes $\Aclass$, where $\delta$ denotes the diameter of such spaces (i.e. they only use distances $\{1,\ldots,\delta\}$) and the other four parameters describe four different families of forbidden triangles (for example, all triangles of odd perimeter smaller than $2K_1$ are forbidden).

Aranda, Bradley-Williams, Hubi{\v c}ka, Karamanlis, Kompatscher, Kone{\v c}n{\'y} and Pawliuk~\cite{Aranda2017,Aranda2017a,Aranda2017c} studied EPPA, Ramsey expansions and (local) SIR's for these classes (see also~\cite{Konecny2018bc,eppatwographs,Konecny2019a}). In particular, if $\Aclass$ is primitive (i.e. it is neither antipodal nor bipartite) and $\delta$ is finite, it can be shown using another result of Hubi\v cka, Kompatscher and Kone\v cn\'y~\cite{Hubickacycles2018} that these (local) stationary independence relations are 1-supported and can be viewed as $\Semig$-shortest path independence relations~\cite{Konecny2018b} with a finite archimedean $\Semig$, which means that Theorem~\ref{thm:cherlin} is a direct consequence of Theorem~\ref{thm:semigroups}.

\section{Conclusion}
We conclude with two questions and a conjecture. The first question is a particular instance of the general question whether 1-supportedness is necessary.
\begin{question}
Consider the structure $\mathbb M_k$ from Example~\ref{ex:non1sup}, that is, the \Fraisse{} limit of all finite $[n]^k$-metric spaces (which are in fact semigroup-valued metric spaces in the sense of Section~\ref{sec:semigroups}). Is the automorphism group of $\mathbb M_k$ simple? (For $k\geq 2$ and $n$ large enough -- if, for example, $n=3$, it is in fact a free amalgamation class, as $(2, \ldots, 2)$ is a free relation.)
\end{question}

The obvious next step is to generalise our results to countable archime\-dean semigroups which do not have to contain a maximum element, thereby obtaining and analogue of Tent and Ziegler's result on the Urysohn space~\cite{Tent2013}. We believe that such a generalisation is quite straightforward. However, there are structures in infinite language which do not even admit a SIR, although they are also very much metric-like. One example is the \emph{sharp Urysohn space}:
\begin{question}
\label{q:sharpurysohn}
Let $\mathbb U^\#$ be the \Fraisse{} limit of the class of all finite complete $\mathbb Q^+$-edge-labelled graphs (here $\mathbb Q^+$ is the set of all positive rational numbers) which contain no triangles $a,b,c$ with $d(a,b)\geq d(a,c)+d(b,c)$ (that is, the triangle inequality is sharp). Is the automorphism group of $\mathbb U^\#$ simple modulo bounded automorphisms?
\end{question}
Note that if we consider $\mathbb N$ instead of $\mathbb Q^+$, the resulting structure can be understood as an $\Semig$-metric space (putting $a\oplus b = a+b-1$ and $a\mleq b$ if $a\leq b$).
\begin{remark}
The sharp Urysohn space is a very peculiar structure, because although it does not admit a SIR, it has EPPA, APA and it is Ramsey when equipped with a (free) linear order.
\end{remark}

The following conjecture and question are closely related to a conjecture from~\cite{Konecny2018b}.	
\begin{conjecture}
Every countable homogeneous complete $L$-edge-labelled graph with $2\leq \lvert L\rvert < \infty$, primitive automorphism group and trivial algebraic closure admits a metric-like SIR.
\end{conjecture}

\begin{question}\label{q:q}
Assume that $\F$ is a transitive countable structure with a metric-like SIR $\ind$ such that $\tp(ab)=\tp(ba)$ for every $a,b\in \F$. Can one define a \pocs{} $\Semig$ on the 2-types of $\F$ such that $\ind$ is the $\Semig$-shortest path independence relation? If the answer is yes, is it true that for every $a\neq b\neq c\in \F$ it holds that $\tp(ab)\mleq \tp(ac)\oplus\tp(bc)$?
\end{question}

The obvious special cases of Question~\ref{q:q} are for finitely many 2-types, 1-supported $\ind$, bounded $\ind$, and their combinations. It is not true that the conditions of Question~\ref{q:q} imply that the structure at hand is an $\Semig$-metric space in the sense of~\cite{Konecny2018b,Hubicka2017sauer}. For example, suppose that $\F$ is the \Fraisse{} limit of the class of all $[n]^1$-metric spaces which also contain a ternary relation $R$ such that if $(a,b,c)\in R$, then $d(a,b)=d(b,c)=d(c,a)=1$. The standard $([n],+,\leq)$-shortest path independence relation is the desired SIR on $\F$.

\bibliography{ramsey.bib}

\newcommand{\etalchar}[1]{$^{#1}$}
\begin{thebibliography}{ABWH{\etalchar{+}}17b}

\bibitem[ABWH{\etalchar{+}}17a]{Aranda2017a}
Andres Aranda, David Bradley-Williams, Eng~Keat Hng, Jan Hubi{\v c}ka,
  Miltiadis Karamanlis, Michael Kompatscher, Mat{\v e}j Kone{\v c}n{\'y}, and
  Micheal Pawliuk.
\newblock Completing graphs to metric spaces.
\newblock {\em Electronic Notes in Discrete Mathematics}, 61:53--60, 2017.
\newblock The European Conference on Combinatorics, Graph Theory and
  Applications (EUROCOMB'17).

\bibitem[ABWH{\etalchar{+}}17b]{Aranda2017c}
Andres Aranda, David Bradley-Williams, Eng~Keat Hng, Jan Hubi{\v c}ka,
  Miltiadis Karamanlis, Michael Kompatscher, Mat{\v e}j Kone{\v c}n{\'y}, and
  Micheal Pawliuk.
\newblock Completing graphs to metric spaces.
\newblock arXiv:1706.00295, accepted to Contributions to Discrete Mathematics,
  2017.

\bibitem[ABWH{\etalchar{+}}17c]{Aranda2017}
Andres Aranda, David Bradley-Williams, Jan Hubi{\v c}ka, Miltiadis Karamanlis,
  Michael Kompatscher, Mat{\v e}j Kone{\v c}n{\'y}, and Micheal Pawliuk.
\newblock Ramsey expansions of metrically homogeneous graphs.
\newblock Accepted to European Journal of Combinatorics, arXiv:1707.02612,
  2017.

\bibitem[Bau16]{Baudisch2016}
Andreas Baudisch.
\newblock Free amalgamation and automorphism groups.
\newblock {\em The Journal of Symbolic Logic}, 81(3):936–947, 2016.

\bibitem[Bra17]{Sam}
Samuel Braunfeld.
\newblock Ramsey expansions of {$\Lambda$}-ultrametric spaces.
\newblock arXiv:1710.01193, 2017.

\bibitem[Che98]{Cherlin1998}
Gregory Cherlin.
\newblock {\em The Classification of Countable Homogeneous Directed Graphs and
  Countable Homogeneous $N$-tournaments}.
\newblock Number 621 in Memoirs of the American Mathematical Society. American
  Mathematical Society, 1998.

\bibitem[Che11]{Cherlin2011b}
Gregory Cherlin.
\newblock Two problems on homogeneous structures, revisited.
\newblock {\em Model theoretic methods in finite combinatorics}, 558:319--415,
  2011.

\bibitem[Che17]{Cherlin2013}
Gregory Cherlin.
\newblock Homogeneous ordered graphs and metrically homogeneous graphs.
\newblock Submitted, December 2017.

\bibitem[CKT20]{calderoni2020simplicity}
F.~Calderoni, A.~Kwiatkowska, and K.~Tent.
\newblock Simplicity of the automorphism groups of order and tournament
  expansions of homogeneous structures.
\newblock Preprint, arXiv:1808.05177, 2020.

\bibitem[Con19]{Conant2015}
Gabriel Conant.
\newblock Extending partial isometries of generalized metric spaces.
\newblock {\em Fundamenta Mathematicae}, 244:1--16, 2019.

\bibitem[DLPS07]{Delhomme2007}
Christian Delhomm{\'e}, Claude Laflamme, Maurice Pouzet, and Norbert~W. Sauer.
\newblock Divisibility of countable metric spaces.
\newblock {\em European Journal of Combinatorics}, 28(6):1746--1769, 2007.

\bibitem[EGT16]{Evans2016}
David~M. Evans, Zaniar Ghadernezhad, and Katrin Tent.
\newblock Simplicity of the automorphism groups of some hrushovski
  constructions.
\newblock {\em Annals of Pure and Applied Logic}, 167(1):22--48, 2016.

\bibitem[EHKN20]{eppatwographs}
David~M. Evans, Jan Hubi{\v{c}}ka, Mat{\v {e}}j Kone{\v {c}}n{\'{y}}, and
  Jaroslav Ne\v{s}et\v{r}il.
\newblock E{P}{P}{A} for two-graphs and antipodal metric spaces.
\newblock {\em Proceedings of the American Mathematical Society},
  148:1901--1915, 2020.

\bibitem[HKK18]{Hubickacycles2018}
Jan Hubi{\v c}ka, Michael Kompatscher, and Mat{\v e}j Kone{\v c}n{\'y}.
\newblock Forbidden cycles in metrically homogeneous graphs.
\newblock Submitted, arXiv:1808.05177, 2018.

\bibitem[HKN17]{Hubicka2017sauerconnant}
Jan Hubi{\v{c}}ka, Mat{\v{e}}j Kone{\v{c}}n{\'y}, and Jaroslav
  Ne{\v{s}}et{\v{r}}il.
\newblock Conant's generalised metric spaces are {R}amsey.
\newblock Accepted to Contributions to Discrete Mathematics, arXiv:1710.04690,
  2017.

\bibitem[HKN18]{Hubicka2017sauer}
Jan Hubi{\v{c}}ka, Mat{\v {e}}j Kone{\v {c}}n{\'{y}}, and Jaroslav
  Ne\v{s}et\v{r}il.
\newblock Semigroup-valued metric spaces: {R}amsey expansions and {E}{P}{P}{A}.
\newblock In preparation, 2018.

\bibitem[HKN19]{Hubicka2018EPPA}
Jan Hubi{\v{c}}ka, Mat\v{e}j Kone\v{c}n\'{y}, and Jaroslav Ne\v{s}et\v{r}il.
\newblock All those {E}{P}{P}{A} classes (strengthenings of the
  {H}erwig-{L}ascar theorem).
\newblock Submitted, arXiv:1902.03855, 2019.

\bibitem[HKNS20]{HubickaSauerSmetric}
Jan Hubi{\v{c}}ka, Mat\v{e}j Kone\v{c}n\'{y}, Jaroslav Ne\v{s}et\v{r}il, and
  Norbert Sauer.
\newblock Ramsey expansions and {E}{P}{P}{A} for {$S$}-metric spaces.
\newblock To appear, 2020.

\bibitem[HN19]{Hubicka2016}
Jan Hubi{\v{c}}ka and Jaroslav Ne\v{s}et\v{r}il.
\newblock All those {R}amsey classes ({R}amsey classes with closures and
  forbidden homomorphisms).
\newblock {\em Advances in Mathematics}, 356C:106791, 2019.

\bibitem[Kon18]{Konecny2018bc}
Mat{\v e}j Kone{\v c}n{\'y}.
\newblock Combinatorial properties of metrically homogeneous graphs.
\newblock Bachelor's thesis, Charles University, 2018.
\newblock arXiv:1805.07425.

\bibitem[Kon19]{Konecny2018b}
Mat{\v e}j Kone{\v c}n{\'y}.
\newblock Semigroup-valued metric spaces.
\newblock Master's thesis, Charles University, 2019.
\newblock arXiv:1810.08963.

\bibitem[Kon20]{Konecny2019a}
Mat{\v e}j Kone{\v c}n{\'y}.
\newblock Extending partial isometries of antipodal graphs.
\newblock {\em Discrete Mathematics}, 343(1):111633, 2020.

\bibitem[KPR18]{Kabil2018}
Mustapha Kabil, Maurice Pouzet, and Ivo~G. Rosenberg.
\newblock Free monoids and generalized metric spaces.
\newblock {\em European Journal of Combinatorics}, 2018.

\bibitem[KS19]{Kaplan2019}
Itay Kaplan and Pierre Simon.
\newblock Automorphism groups of finite topological rank.
\newblock {\em Transactions of the American Mathematical Society},
  372(3):2011--2043, 2019.

\bibitem[Li18]{Li2018}
Yibei Li.
\newblock Simplicity of the automorphism groups of some binary homogeneous
  structures determined by triangle constraints.
\newblock {\em arXiv:1806.01671}, 2018.

\bibitem[Li19]{Li2019}
Yibei Li.
\newblock Automorphism groups of homogeneous structures with stationary weak
  independence relations.
\newblock {\em arXiv:1911.08540}, 2019.

\bibitem[Li20]{Li2020}
Yibei Li.
\newblock Automorphism groups of linearly ordered homogeneous structures.
\newblock {\em arXiv:2009.02475}, 2020.

\bibitem[MT11]{Macpherson2011b}
Dugald Macpherson and Katrin Tent.
\newblock Simplicity of some automorphism groups.
\newblock {\em Journal of Algebra}, 342(1):40--52, 2011.

\bibitem[NVT09]{NVT2009}
Lionel Nguyen Van~Th{\'e}.
\newblock Ramsey degrees of finite ultrametric spaces, ultrametric {U}rysohn
  spaces and dynamics of their isometry groups.
\newblock {\em European Journal of Combinatorics}, 30(4):934--945, 2009.

\bibitem[NVT15]{NVT14}
Lionel Nguyen Van~Th{\'e}.
\newblock {A survey on structural {R}amsey theory and topological dynamics with
  the {K}echris-{P}estov-{T}odorcevic correspondence in mind}.
\newblock {\em Selected Topics in Combinatorial Analysis}, 17(25):189--207,
  2015.

\bibitem[Sau12]{Sauer2012}
Norbert~W. Sauer.
\newblock Vertex partitions of metric spaces with finite distance sets.
\newblock {\em Discrete Mathematics}, 312(1):119--128, 2012.

\bibitem[Sau13a]{Sauer2013}
Norbert~W. Sauer.
\newblock Distance sets of {U}rysohn metric spaces.
\newblock {\em Canadian Journal of Mathematics}, 65(1):222--240, 2013.

\bibitem[Sau13b]{sauer2013b}
Norbert~W. Sauer.
\newblock Oscillation of {U}rysohn type spaces.
\newblock In {\em Asymptotic geometric analysis}, pages 247--270. Springer,
  2013.

\bibitem[Sin17]{Siniora2}
Daoud Siniora.
\newblock {\em Automorphism Groups of Homogeneous Structures}.
\newblock PhD thesis, University of Leeds, March 2017.

\bibitem[TZ13a]{Tent2013b}
Katrin Tent and Martin Ziegler.
\newblock The isometry group of the bounded urysohn space is simple.
\newblock {\em Bulletin of the London Mathematical Society}, 45(5):1026--1030,
  2013.

\bibitem[TZ13b]{Tent2013}
Katrin Tent and Martin Ziegler.
\newblock On the isometry group of the {U}rysohn space.
\newblock {\em Journal of the London Mathematical Society}, 87(1):289--303,
  2013.

\end{thebibliography}
\end{document}